\documentclass[11pt]{article}

\usepackage{ascmac}
\usepackage{amsthm}
\usepackage{amsmath}
\usepackage{amssymb}
\usepackage{amsfonts}
\usepackage{bm}

\newtheorem{thm}{Theorem}[section]
\newtheorem{lem}[thm]{Lemma}

\newtheorem{cor}[thm]{Corollary}

\makeatletter
\@addtoreset{equation}{section}

\makeatother

\title{Zeros of a polynomial of $\zeta^{(j)}(s)$}
\author{TOMOKAZU ONOZUKA}

\begin{document}
\baselineskip=17pt
\date{}
\maketitle
\renewcommand{\thefootnote}{}
\renewcommand{\thefootnote}{\arabic{footnote}}
\setcounter{footnote}{0}
\begin{abstract}
We give results on  zeros of a polynomial of $\zeta(s),\zeta'(s),\ldots,\zeta^{(k)}(s)$. First, we give a zero free region and prove that there exist zeros corresponding to the trivial zeros of the Riemann zeta function. Next, we estimate the number of zeros whose imaginary part is in $(1,T)$. Finally, we study the distribution of the real part and the imaginary part of zeros, respectively.
\end{abstract}
\section{Introduction}

The Riemann zeta function $\zeta(s)$ is one of the most important functions in number theory, and its importance comes from its relation to the distribution of primes. The theory of the Riemann zeta function has a famous conjecture, which is the Riemann hypothesis. The Riemann hypothesis states that all of the nontrivial zeros of the Riemann zeta function are located on the critical line $\Re (s)=1/2$. If the Riemann hypothesis is true, many conjectures in number theory hold. Therefore a lot of mathematicians study the zeros of the Riemann zeta function. 

Derivatives of the Riemann zeta function are also important because of the relation between its nontrivial zeros and the Riemann hypothesis. Speiser \cite{spe} proved that the Riemann hypothesis is equivalent to $\zeta'(s)$ having no zeros in $0<\Re(s)<1/2$. Levinson and Montgomery \cite{lm} proved that the Riemann hypothesis implies that $\zeta^{(k)}(s)$ has at most a finite number of non-real zeros in $\Re(s)<1/2$ for $k\geq1$. In the case $k=2,3$, Yildirim \cite{yi} showed that the Riemann hypothesis implies that  $\zeta''(s)$ and $\zeta'''(s)$ have no  zeros in the strip $0\leq\Re(s)<1/2$.

Zeros of derivatives of the Riemann zeta function was investigated by Berndt \cite{be} and Levinson and Montgomery \cite{lm}. Berndt \cite{be} gave a formula of the Riemann-von Mangoldt type for $\zeta^{(k)}(s)$ with $k\geq1$. 
For function $f(s)$, let $N_{f(s)}(T_1,T_2)$ be the number of zeros of $f(s)$ counted with multiplicity in the region $T_1<\Im s<T_2$.
Then the Riemann-von Mangoldt formula states that for sufficiently large T, we have
\begin{align}
N_{\zeta(s)}(0,T)=
  \displaystyle\frac{T}{2\pi}\log\frac{T}{2\pi}-\frac{T}{2\pi}+O(\log T).\label{RVF1}
\end{align}
Berndt \cite{be} generalized this formula for $\zeta^{(k)}(s)$ with $k\geq1$;
\begin{align}
N_{\zeta^{(k)}(s)}(0,T)=
  \displaystyle\frac{T}{2\pi}\log\frac{T}{4\pi}-\frac{T}{2\pi}+O(\log T).\label{RVF2}
\end{align}
Hence the number of zeros of $\zeta^{(k)}(s)$ is asymptotically equal to the number of zeros of $\zeta(s)$ in $0<\Im(s)<T$.

One of the most interesting objects of $\zeta(s)$ is the distribution of the real part of nontrivial zeros. In 1914, Bohr and Landau \cite{bl} showed that most of nontrivial zeros of $\zeta(s)$ lie near the critical line. The real part of zeros of  $\zeta^{(k)}(s)$ was also studied by Levinson and Montgomery \cite{lm}. They showed 
\begin{align}
\notag &2\pi\sum_{T<\gamma^{(k)}<T+U}\left(\beta^{(k)}-\frac{1}{2}\right)\\
&\qquad=kU\log\log\frac{T}{2\pi}-U\log\frac{(\log2)^k}{2^{1/2}}+O\left(\frac{U^2}{T\log T}\right)+O\left(\log T\right)\label{zeroright}
\end{align}
for $0<U<T$ and $k\geq1$, where $\rho^{(k)}=\beta^{(k)}+i\gamma^{(k)}$ denotes the zeros of $\zeta^{(k)}(s)$. This estimate implies that zeros of $\zeta^{(k)}(s)$ would be mainly located in the right half plane $\Re(s)>1/2$. They also considered zeros near the critical line. Let $N_k^-(c,T)$ denote the number of zeros of $\zeta^{(k)}(s)$ in $0<\Im (s)<T$, $\Re(s)<c$. Similarly, $N_k^+(c,T)$ denotes the number of zeros of $\zeta^{(k)}(s)$ in $0<\Im (s)<T$, $\Re(s)>c$. Then, they \cite{lm} estimated
\begin{align}
N_k^+\left(\frac{1}{2}+\delta,T\right)+N_k^-\left(\frac{1}{2}-\delta,T\right)=O\left(\frac{T\log\log T}{\delta}\right),\label{T3}
\end{align}
for $\delta>0$ uniformly. When $\delta=(\log\log T)^2/\log T$, by (\ref{RVF2}), we have
\begin{align*}
N_k^+\left(\frac{1}{2}+\delta,T\right)+N_k^-\left(\frac{1}{2}-\delta,T\right)=O\left(\frac{N_{\zeta^{(k)}(s)}(0,T)}{\log\log T}\right).
\end{align*}
Hence similar to zeros of $\zeta(s)$, most of zeros of $\zeta^{(k)}(s)$ are also clustered around $\Re(s)=1/2$.

On the other hand, the imaginary part of zeros of $\zeta^{(k)}(s)$ was also investigated. 
In the case $k=0$, for $x>1$, Landau \cite{la} proved
\begin{align}
\sum_{0<\gamma^{(0)}<T}x^{\rho^{(0)}}=-\Lambda(x)\frac{T}{2\pi}+O(\log T)=o(N_{\zeta(s)}(0,T)),\label{Xlandau}
\end{align}
where $\Lambda(x)$ is the von Mangoldt $\Lambda$ function if $x$ is an integer, and otherwise $\Lambda(x)=0$.
Since most of $\Re (\rho^{(0)})$ are close to $1/2$, we have roughly
\begin{align*}
x^{1/2}\sum_{0<\gamma^{(0)}<T}\exp(i\gamma^{(0)}\log x)\approx \sum_{0<\gamma^{(0)}<T}x^{\rho^{(0)}}\approx o(N_{\zeta(s)}(0,T))
\end{align*}
for any $x>1$.
By this estimate, $\{\gamma^{(0)}$ mod $\frac{2\pi}{\log x}\}_{\gamma^{(0)}>0}$ might be regarded as an uniformly distributed sequence.  Actually, it is known that $\{\alpha\gamma^{(0)}\}_{\gamma^{(0)}>0}$ is uniformly distributed modulo one for any non-zero real $\alpha$. This fact was first proved by Rademacher \cite{ra} under the Riemann hypothesis. After then, Elliot \cite{el} remarked that this result holds unconditionally, and Hlawka \cite{hl} finally proved it unconditionally.

Formulas \eqref{RVF1}-\eqref{Xlandau} were generalized by many mathematicians. For $a\in\mathbb{C}$, Landau \cite{bo} gave a formula of the Riemann-von Mangoldt type for $\zeta(s)-a$;
\begin{align}
\label{T1}N_{\zeta(s)-a}(1,T)=\begin{cases}
  \displaystyle\frac{T}{2\pi}\log\frac{T}{2\pi}-\frac{T}{2\pi}+O(\log T)&(a\neq1),\\
\\
  \displaystyle\frac{T}{2\pi}\log\frac{T}{4\pi}-\frac{T}{2\pi}+O(\log T)&(a=1).
\end{cases}
\end{align}
The author \cite{on} considered $\zeta^{(k)}(s)-a$ and gave a formula;
\begin{align*}
N_{\zeta^{(k)}(s)-a}(1,T)=\frac{T}{2\pi}\log\frac{T}{2\pi}-\frac{T}{2\pi}+O(\log T)
\end{align*}
for $k\geq1$ and $a\neq0$. (The case $k=0$ was given by Landau, and the case $a=0$ was given by Berndt as mentioned above.) Furthermore Koutsaki, Tamazyan, and Zaharescu \cite{ko} studied linear combinations of $\zeta^{(j)}(s)$, and they proved
\begin{align*}
N_{f(s)}(0,T)=\frac{T}{2\pi}\log\frac{T}{2\pi}-\frac{T}{2\pi}+O(\log T),
\end{align*}
where $f(s)=c_0\zeta(s)+c_1\zeta'(s)+\cdots+c_k\zeta^{(k)}(s)$ with $c_0,\ldots,c_k\in\mathbb{R}$ and $c_0,c_k\neq0$. 
In addition, Nakamura \cite{na} studied polynomials of derivatives of zeta functions by using the universality theorem, and he proved that for any $1/2<\sigma_1<\sigma_2<1$, there exists a constant $C$ such that $P(s)$ has more than $CT$ zeros in $\sigma_1<\Re(s)<\sigma_2$ and $0<\Im(s)<T$,
where $P(s)$ is a polynomial of derivatives of the Riemann zeta function.
As a generalization of \eqref{zeroright} and \eqref{T3}, Levinson \cite{le} proved that for sufficiently large $T$, $T^{1/2}\leq U \leq T$, $\delta=(\log\log T)^2/\log T$ and $a\in\mathbb{C}$, we have
\begin{align}
N_{\zeta(s)-a}^+\left(\frac{1}{2}+\delta;T,T+U\right)+N_{\zeta(s)-a}^-\left(\frac{1}{2}-\delta;T,T+U\right)=O\left(\frac{U\log\log T}{\delta}\right),\label{levin}
\end{align}
where $N_{f(s)}^+\left(c;T_1,T_2\right)$ and $N_{f(s)}^-\left(c;T_1,T_2\right)$ are defined as the number of zeros of $f(s)$ in $\{s\in\mathbb{C}\ |\ T_1<\Im (s)<T_2,\Re(s)>c\}$ and $\{s\in\mathbb{C}\ |\ T_1<\Im (s)<T_2,\Re(s)<c\}$, respectively. He proved \eqref{levin} with $\delta=(\log\log T)^2/\log T$, in addition, he noted that this result can be generalized for any small $\delta>0$. In the proof of \eqref{levin}, he gave an estimate of the sum $2\pi\sum_{T<\gamma_a<T+U}(\beta_a+b)$ for $b>2$ (see \cite[Lemma 5]{le}), where $\rho_a=\beta_a+i\gamma_a$ denotes the zeros of $\zeta(s)-a$. Calculating
$$
2\pi\sum_{T<\gamma_a<T+U}(\beta_a+b)-2\pi\left(b+\frac{1}{2}\right)N_{\zeta(s)-a}(T,T+U),
$$
we can deduce a generalization of \eqref{zeroright}. As a generalization of \eqref{levin}, the author gave an estimate
 \begin{align}
N_{\zeta^{(k)}(s)-a}^+\left(\frac{1}{2}+\delta;T,T+U\right)+N_{\zeta^{(k)}(s)-a}^-\left(\frac{1}{2}-\delta;T,T+U\right)=O\left(\frac{U\log\log T}{\delta}\right),\label{levin2}
\end{align}
for $\delta=(\log\log T)^2/\log T$, $\alpha>1/2$ and $T^{\alpha}\leq U\leq T$. Let $\rho^{(k)}_a=\beta^{(k)}_a+i\gamma^{(k)}_a$ denote the zeros of $\zeta^{(k)}(s)-a$. Similar to Levinson's proof, in the proof of \eqref{levin2}, the sum  $2\pi\sum_{T<\gamma^{(k)}_a<T+U}(\beta^{(k)}_a+b)$ was estimated for large $b$, so we can also deduce a generalization of \eqref{zeroright} for zeros of $\zeta^{(k)}(s)-a$. Finally, we will see analogues of \eqref{Xlandau}. Steuding \cite{st} proved that for any positive real number $x\neq1$, we have
\begin{align}
\label{T2}\sum_{0<\gamma_a<T}x^{\rho_a}=\left(\alpha(x)-x\Lambda\left(\frac{1}{x}\right)\right)\frac{T}{2\pi}+O(T^{\frac{1}{2}+\varepsilon}),
\end{align}
where $\alpha(x)$ is a coefficient of the series
\begin{align}
\frac{\zeta'(s)}{\zeta(s)-a}=\sum_{d\in2^{-n}\mathbb{N}\ (n\in\mathbb{N})}\frac{\alpha(d)}{d^s}\label{alp}
\end{align}
for $x\in\mathbb{Z}$ and $\alpha(x)=0$ for $x\notin\mathbb{Z}$ if $a\neq1$. If $a=1$,  $\alpha(x)$ is also the coefficient of (\ref{alp}) for $2^nx\in\mathbb{Z}$ with some $n\in\mathbb{N}$. If $2^nx\notin\mathbb{Z}$ for any $n\in\mathbb{N}$,  $\alpha(x)=0$. The author also considered an analogue of \eqref{Xlandau}.
For any $k\in\mathbb{N}$, $a\in\mathbb{C}$ and $x>1$, we have
\begin{align}
\label{Xon}&\sum_{1<\gamma_a^{(k)}<T}x^{\rho_a^{(k)}}\\
&=\begin{cases}
  \displaystyle\frac{T}{2\pi}\sum_{\substack{l\geq0\\n_0,\ldots,n_l\geq2\\x=n_0\cdots n_l}}\frac{(-1)^{k(l+1)}}{a^{l+1}}(\log n_0)^{k+1}(\log n_1\cdots\log n_l)^k+O(\log T)&(a\neq0),\\
  \displaystyle\frac{T}{2\pi}\sum_{\substack{l\geq0\\n_0\geq2\\n_1,\ldots,n_l\geq3\\x=n_0\cdots n_l/2^{l+1}}}\left(\frac{-1}{(\log 2)^k}\right)^{l+1}(\log n_0)^{k+1}(\log n_1\cdots\log n_l)^k+O(\log T)&(a=0).
\end{cases}
\end{align}
If $a\neq0$, the summation of the right-hand side is zero for $x\notin\mathbb{Z}$, and if $a=0$ and $2^nx\notin\mathbb{Z}$ for any $n\in\mathbb{N}$, the summation of the right-hand side is zero. From \eqref{T2} and \eqref{Xon}, $\{\alpha\gamma_a\}$ and $\{\alpha\gamma_a^{(k)}\}$ are uniformly distributed modulo one for any non-zero real $\alpha$, respectively (see \cite{st}\cite{lee}). 

In this paper, we treat a polynomial consisting of $\zeta(s),\zeta'(s),\ldots,\zeta^{(k)}(s)$. For $k,M\in\mathbb{N}$, $d_{lj}\in\mathbb{N}\cup\{0\}$ and $c_j\in\mathbb{C}\setminus\{0\}$, we put
\begin{align*}
F(s):=\sum_{j=1}^M c_j\zeta^{(0)}(s)^{d_{0j}}\zeta^{(1)}(s)^{d_{1j}}\cdots\zeta^{(k)}(s)^{d_{kj}}.
\end{align*} 
We assume that $F(s)$ is not a constant function, that is, at least one of $d_{lj}$ is nonzero. At the International Congress of Mathematicians in 1900, Hilbert pointed out that $\zeta(s)$ does not satisfy any algebraic differential equation of finite order, so $F(s)$ is not a constant function if at least one of $d_{lj}$ is nonzero. We define the first degree of $F(s)$ as
\begin{align*}
\deg_{1}\left(F(s)\right):=\max_{1\leq j\leq M}\sum_{l=0}^k d_{lj}.
\end{align*}
Especially, we have
\begin{align}
\deg_{1}\left(\zeta^{(0)}(s)^{d_{0j}}\zeta^{(1)}(s)^{d_{1j}}\cdots\zeta^{(k)}(s)^{d_{kj}}\right):=\sum_{l=0}^k d_{lj}.\label{deg1}
\end{align}
The second degree of $F(s)$ is defined by
\begin{align*}
&\deg_{2}\left(F(s)\right)\\
&:=\max\left\{\sum_{l=0}^k ld_{lj}\ \Bigg|\ \deg_{1}\left(\zeta^{(0)}(s)^{d_{0j}}\cdots\zeta^{(k)}(s)^{d_{kj}}\right)=\deg_{1}\left(F(s)\right),1\leq j\leq M\right\}.
\end{align*} 
By \eqref{deg1}, the definition $\deg_2$ can be rewritten as
\begin{align*}
\deg_{2}\left(F(s)\right)=\max\left\{\sum_{l=0}^k ld_{lj}\ \Bigg|\ \sum_{l=0}^k d_{lj}=\deg_{1}\left(F(s)\right),1\leq j\leq M\right\}
\end{align*} 
Similar to \eqref{deg1}, we have
\begin{align*}
\deg_{2}\left(\zeta^{(0)}(s)^{d_{0j}}\zeta^{(1)}(s)^{d_{1j}}\cdots\zeta^{(k)}(s)^{d_{kj}}\right):=\sum_{l=0}^k ld_{lj}.
\end{align*}
In order to obtain a functional equation for $F(s)$ (see Lemma \ref{lem1}), we need an assumption
\begin{align}
\sum_{j\in J}c_j\neq0,\label{condition}
\end{align}
where $J$ is defined by
\begin{align*}
J:=\left\{j\in[1,M]\ \Bigg|\  \begin{array}{l}\deg_{1}\left(\zeta^{(0)}(s)^{d_{0j}}\cdots\zeta^{(k)}(s)^{d_{kj}}\right)=\deg_{1}\left(F(s)\right),\\ \deg_{2}\left(\zeta^{(0)}(s)^{d_{0j}}\cdots\zeta^{(k)}(s)^{d_{kj}}\right)=\deg_{2}\left(F(s)\right)\end{array}\right\}.
\end{align*}
We study zeros of $F(s)$. Hereafter we set $s=\sigma+it$ and $\rho_F=\beta_F+i\gamma_F$ denotes the zeros of $F(s)$. The first result gives a zero free region for $F(s)$.
\begin{thm}\label{thm1}
Let $F(s)$ satisfy \eqref{condition}. For $\varepsilon>0$, there exist real numbers $E_{1F\varepsilon}=E_{1F}$ and $E_{2F}$ such that $F(s)\neq0$ on $$\{s\in\mathbb{C}\ |\ \sigma\leq E_{1F},|s+2n|\geq\varepsilon\ (n\in\mathbb{N})\}\cup\{s\in\mathbb{C}\ |\ \sigma\geq E_{2F}\}.$$
\end{thm}
The second result gives zeros near the real axis. We define the region $\mathcal{C}_{n,\varepsilon}$ as
\begin{align*}
\mathcal{C}_{n,\varepsilon}:=\{s\in\mathbb{C}\ |\ |s+2n|<\varepsilon\}.
\end{align*}
\begin{thm}\label{thm3}
Let $F(s)$ satisfy \eqref{condition}. For any $\varepsilon>0$, there exists a positive integer $N=N_{F,\varepsilon}$ such that $F(s)$ has exactly $\deg_1(F(s))$ zeros in $\mathcal{C}_{n,\varepsilon}$ for each $n\geq N$. 
\end{thm}
Spira \cite{sp1} proved that there is a $C_k$ such that $\zeta^{(k)}(s)$ has exactly one real zero in $(-1-2n,1-2n)$ with $1-2n\leq C_k$. Levinson \cite{le} pointed out that $\zeta(s)=a$ has exactly one root in the neighborhood of $s=-2n$ for large $n$. Theorem \ref{thm3} is a generalization of these results. The third theorem counts the number of zeros of $F(s)$ in $1<t<T$. To state the theorem, we define $n_F$ as follows;
$$
F(s)=\sum_{n=n_F}^\infty \frac{\eta_n}{n^s}\qquad(\eta_{n_F}\neq0).
$$
Note that $F(s)$ can be expressed as a Dirichlet series since every $\zeta^{(j)}(s)$ has the Dirichlet series expression $\sum_{n\geq1}(-\log n)^j/n^s$. (In this paper, we define $(-\log 1)^0=1$.)
\begin{thm}\label{1}
Let $F(s)$ satisfy \eqref{condition}. For large $T$, we have
\begin{align*}
N_{F(s)}(1,T)=\frac{\deg_1(F(s))T}{2\pi}\log\frac{T}{2\pi e}-\frac{T}{2\pi}\log n_F+O(\log T).
\end{align*}
\end{thm}
The fourth and the fifth results describe the real part of zeros of $F(s)$.
\begin{thm}\label{thm4}
Let $F(s)$ satisfy \eqref{condition}. For large $T$, $\alpha>1/2$ and $T^\alpha\leq U\leq T$, we have
\begin{align}
\notag &2\pi\sum_{T<\gamma_F<T+U}\left(\beta_F-\frac{1}{2}\right)\\
&\qquad=\deg_2(F(s))U\log\log T+U\log\left|\frac{\sum_{j\in J}c_j}{\eta_{n_F}/n_F^{1/2}}\right|+O\left(\frac{U}{\log T}\right).\label{leftright}
\end{align}
\end{thm}
\begin{thm}\label{thm5}
Let $F(s)$ satisfy \eqref{condition}. For large $T$, $\alpha>1/2$ and $T^\alpha\leq U\leq T$, we have
 \begin{align}
N_{F(s)}^+\left(\frac{1}{2}+\delta;T,T+U\right)+N_{F(s)}^-\left(\frac{1}{2}-\delta;T,T+U\right)=O\left(\frac{U\log\log T}{\delta}\right),
\end{align}
for $\delta>0$ uniformly.
\end{thm}
From Theorem \ref{thm4}, if $\deg_2(F(s))$ is positive,  zeros of $F(s)$ would be mainly located in the right half plane $\Re(s)>1/2$. If $\deg_2(F(s))=0$, the distribution of the zeros depends on $|\sum_{j\in J}c_j|$ and $|\eta_{n_F}/n_F^{1/2}|$. If $|\sum_{j\in J}c_j|>|\eta_{n_F}/n_F^{1/2}|$, the zeros would be mainly located in $\Re(s)>1/2$.  If $|\sum_{j\in J}c_j|<|\eta_{n_F}/n_F^{1/2}|$, the zeros would be mainly located in $\Re(s)<1/2$. Finally if $|\sum_{j\in J}c_j|=|\eta_{n_F}/n_F^{1/2}|$, the main terms of \eqref{leftright} vanish, so we can not obtain the sign of $\sum(\beta_F-1/2)$.

By Theorem \ref{thm5}, most of zeros of $F(s)$ are close to the critical line. Although ``nontrivial zeros of $F(s)$" are generally not on the critical line, $F(s)$ has the property of the Riemann zeta function whose nontrivial zeros are close to the  critical line. (Here ``nontrivial zeros of $F(s)$" means that the zeros of $F(s)$ with $|\gamma_F|>1$.)

Next, we see the main results on the imaginary parts of zeros of $F(s)$. 
\begin{thm}\label{6}
Let $F(s)$ satisfy \eqref{condition}. For real number $x>1$ and large $T$, we have
\begin{align*}
\sum_{1<\gamma_{F}<T}x^{\rho_F}=\left(\alpha_{(F'/F)(s)}(x)\right)\frac{T}{2\pi}+O(\log T),
\end{align*}
where $\alpha_{(F'/F)(s)}(x)$ is a coefficient of the series
\begin{align}
\frac{F'}{F}(s)=\sum_{d\in n_F^{-n}\mathbb{N}\ (n\in\mathbb{N})}\frac{\alpha_{(F'/F)(s)}(d)}{d^s}\label{alpFF}
\end{align}
for $x\in n_F^{-n}\mathbb{N}$ with some  $n\in\mathbb{N}$ and $\alpha_{(F'/F)(s)}(x)=0$ for $x\notin  n_F^{-n}\mathbb{N}$.
\end{thm}
When $F(s)=\zeta^{(k)}(s)-a$, the series \eqref{alpFF} is given in \cite[Lemma 4.1]{on}. Similar to the proof of \cite[Lemma 4.1]{on}, we can also give the series \eqref{alpFF} for any $F(s)$. From this theorem, we can deduce the following corollary.

\begin{cor}\label{cor7}
Let $F(s)$ satisfy \eqref{condition}. The sequence $\{\alpha\gamma_F\}_{\gamma_F>1}$ is uniformly distributed modulo one for any non-zero real $\alpha$.
\end{cor}


\section{Lemmas and Fundamental Results}

In this section, we prove some lemmas and fundamental results on the zeros of $F(s)$. Theorem \ref{thm1} and Theorem \ref{thm3} are proved in this section. 
\begin{lem}\label{lemma1}
For any $\varepsilon>0$, we have $\eta_n=O(n^\varepsilon).$
\end{lem}
\begin{proof}
For sufficiently large $\sigma$, we have
\begin{align*}
&c_j\zeta^{(0)}(s)^{d_{0j}}\zeta^{(1)}(s)^{d_{1j}}\cdots\zeta^{(k)}(s)^{d_{kj}}\\
&\quad=c_j\left(\sum_{n=1}^{\infty}\frac{(-\log n)^0}{n^s}\right)^{d_{0j}}\left(\sum_{n=1}^{\infty}\frac{(-\log n)^1}{n^s}\right)^{d_{1j}}\cdots\left(\sum_{n=1}^{\infty}\frac{(-\log n)^k}{n^s}\right)^{d_{kj}}\\
&\quad=:\sum_{n=1}^{\infty}\frac{\eta^{(j)}_n}{n^s}.
\end{align*}
Since $\log n=O(n^{\varepsilon})$ holds for any small $\varepsilon>0$, the coefficient $\eta^{(j)}_n$ can be estimated as
\begin{align*}
\left|\eta^{(j)}_n\right|\ll n^\varepsilon \sum_{n_1n_2\cdots n_{d_{0j}+\cdots+d_{kj}}=n}1\ll n^\varepsilon d(n)^{d_{0j}+\cdots+d_{kj}},
\end{align*}
where $d(n)$ is the number of the divisors of $n$. It follows from $d(n)=O(n^{\varepsilon})$ that we have $\eta^{(j)}_n=O(n^\varepsilon)$. Hence we have $\eta_n=\eta^{(1)}_n+\cdots+\eta^{(M)}_n=O(n^{\varepsilon})$.
\end{proof}

\begin{lem}\label{lem2.2}
For sufficiently large $\sigma$, we have
\begin{align*}
F(s)=\frac{\eta_{n_F}}{n_F^s}+O((n_F+1)^{-\sigma}).
\end{align*}
\end{lem}
\begin{proof}
Since
\begin{align*}
F(s)=\frac{\eta_{n_F}}{n_F^s}+\frac{\eta_{n_F+1}}{(n_F+1)^s}+\sum_{n>n_F+1}\frac{\eta_n}{n^s},
\end{align*}
it is enough to prove that the last term can be bounded by $O((n_F+1)^{-\sigma}).$ By Lemma \ref{lemma1}, we have
\begin{align*}
\sum_{n>n_F+1}\frac{\eta_n}{n^s}\ll \int_{n_F+1}^\infty x^{-\sigma+\varepsilon}dx\ll (n_F+1)^{-\sigma}
\end{align*}
because of the estimate $ (n_F+1)^{1+\varepsilon}=O(1)$.
\end{proof}

\begin{lem}\label{lem1}
Let $F(s)$ satisfy \eqref{condition}. For any $c>1$ and $\varepsilon>0$, the following equation holds in the region $\{s\in\mathbb{C}\ |\ \sigma>c,|s-(2n-1)|\geq\varepsilon \ (n\in\mathbb{N})\}$;
\begin{align*}
&F(1-s)\\&=\left(\sum_{j\in J}c_j\right)(-\log s)^{\deg_{2}\left(F(s)\right)}\left\{2(2\pi)^{-s}\Gamma(s)\cos\frac{\pi s}{2}\ \zeta(s)\right\}^{\deg_{1}\left(F(s)\right)}\left(1+O\left(\frac{1}{|\log s|}\right)\right).
\end{align*}
\end{lem}

\begin{proof}
By \cite[Lemma 2.1]{on}, we have the  functional equation for $\zeta^{(l)}(s)$
\begin{align*}
\zeta^{(l)}(1-s)&=(-\log s)^l2(2\pi)^{-s}\Gamma(s)\cos\frac{\pi s}{2}\ \zeta(s)\left(1+O\left(\frac{1}{|\log s|}\right)\right).
\end{align*}
We note that this equation holds not only in $\{s\in\mathbb{C}\ |\ \sigma>c,|t|\geq1\}$ but also in $\{s\in\mathbb{C}\ |\ \sigma>c,|s-(2n-1)|\geq\varepsilon \ (n\in\mathbb{N})\}$, since  \cite[(8)]{on} is valid in the latter region. Multiplying this equation several times, we have
\begin{align*}
&\zeta^{(0)}(1-s)^{d_{0j}}\zeta^{(1)}(1-s)^{d_{1j}}\cdots\zeta^{(k)}(1-s)^{d_{kj}}\\
&\quad=(-\log s)^{\sum_{l=0}^k ld_{lj}}\left\{2(2\pi)^{-s}\Gamma(s)(\log s)^k\cos\frac{\pi s}{2}\ \zeta(s)\right\}^{\sum_{l=0}^k d_{lj}}\left(1+O\left(\frac{1}{|\log s|}\right)\right).
\end{align*}
By Stirling's formula , we estimate
\begin{align*}
&\zeta^{(0)}(1-s)^{d_{0j}}\zeta^{(1)}(1-s)^{d_{1j}}\cdots\zeta^{(k)}(1-s)^{d_{kj}}\\
&\quad\asymp |\log s|^{\sum_{l=0}^k ld_{lj}}\left\{\left|\frac{s}{2\pi e}\right|^{\sigma-1/2}\exp\left(|t|\left(\frac{\pi}{2}-|\arg(s)|\right)\right)\right\}^{\sum_{l=0}^k d_{lj}}
\end{align*}
where $\arg(s)\in(-\pi/2,\pi/2)$. Thus the main term of $F(s)$ is the terms whose index is in $J$.
\end{proof}

\noindent{\bf (Proof of Theorem \ref{thm1})}\ 
By Lemma \ref{lem1}, we set
\begin{align*}
&F(1-s)\\&=\left(\sum_{j\in J}c_j\right)(-\log s)^{\deg_{2}\left(F(s)\right)}\left\{2(2\pi)^{-s}\Gamma(s)\cos\frac{\pi s}{2}\ \zeta(s)\right\}^{\deg_{1}\left(F(s)\right)}\left(1+O\left(\frac{1}{|\log s|}\right)\right)\\
&=:A_1(s)+A_2(s).
\end{align*}
We can easily check $A_1(s)\neq0$ in  $\{s\in\mathbb{C}\ |\ \sigma>c,|s-(2n-1)|\geq\varepsilon \ (n\in\mathbb{N})\}$. It follows from $A_2(s)=O(|A_1(s)/\log s|)$ that $|A_2(s)|<|A_1(s)|$ holds for sufficiently large $\sigma$. Hence there exists $E_{1F\varepsilon}=E_{1F}$ such that $F(1-s)=A_1(s)+A_2(s)\neq0$ for $\sigma\geq 1-E_{1F}$ and $|s-(2n-1)|\geq\varepsilon\ (n\in\mathbb{N})$.

By Lemma \ref{lem2.2}, we set
$$
F(s)=\frac{\eta_{n_F}}{n_F^s}+O\left((n_F+1)^{-\sigma}\right)=:B_1+B_2
$$
for large $\sigma$. There exists $E_{2F}$ such that $|B_1|>|B_2|$ holds for $\sigma\ge E_{2F}$. Because of $B_1\neq0$, $\{s\in\mathbb{C}\ |\ \sigma\ge E_{2F}\}$ is a zero free region for $F(s)$.  {\hfill $\square$}
\ \\

\noindent{\bf (Proof of Theorem \ref{thm3})}\ 
In the proof of Theorem \ref{thm1}, we checked $|A_2(s)|<|A_1(s)|$ for large $\sigma$ and $|s-(2n-1)|\geq\varepsilon \ (n\in\mathbb{N})$. Applying Rouch\'e's theorem, $F(1-s)$ and $A_1(s)$ has the same number of zeros in $\{s\in\mathbb{C}\ |\ |s-(2n-1)|<\varepsilon\}$ for large $n\in \mathbb{N}$. The function $A_1(s)$ has exactly $\deg_{1}\left(F(s)\right)$ zeros at $s=2n-1$. Therefore, $F(1-s)$ has exactly $\deg_{1}\left(F(s)\right)$ zeros in $\{s\in\mathbb{C}\ |\ |s-(2n-1)|<\varepsilon\}$.{\hfill $\square$}
\ \\

\begin{lem}\label{LEM4}
There exist complex numbers $A_{F},\ B_{F}$ and non-negative integers $N_F,m_{F}$ such that the following equations hold;
\begin{align*}
&(s-1)^{N_F}F(s)=e^{A_{F}+B_{F}s}s^{m_{F}}\prod_{\substack{\rho_F\neq0\\ \rho_F\mbox{\footnotesize{:\ zeros of }}F(s)}}\left(1-\frac{s}{\rho_F}\right)e^{s/\rho_F},\\
&\frac{F'(s)}{F(s)}=-\frac{N_F}{s-1}+B_{F}+\frac{m_{F}}{s}+\sum_{\substack{\rho_F\neq0\\ \rho_F\mbox{\footnotesize{:\ zeros of }}F(s)}}\left(\frac{1}{s-\rho_F}+\frac{1}{\rho_F}\right).
\end{align*}
\end{lem}
\begin{proof}
Since
\begin{align*}
\zeta^{(k)}(s)=\frac{k!}{2\pi i}\int_{|z-s|=a}\frac{\zeta(z)}{(z-s)^{k+1}}dz
\end{align*}
and
\begin{align*}
\zeta(s)&\ll\exp(|s|^{1+\varepsilon}),
\end{align*}
we have
\begin{align*}
\zeta^{(k)}(s)&\ll\exp(|s|^{1+\varepsilon})
\end{align*}
for any small $\varepsilon>0$ and large $|s|$. Therefore we also have
\begin{align*}
F(s)&\ll\exp(|s|^{1+\varepsilon}).
\end{align*}

$F(s)$ is holomorphic on $\mathbb{C}\setminus\{1\}$, because of the continuation of $\zeta(s)$. We assume that $F(s)$ has a pole of order $N_F$ at $s=1$. Then $(s-1)^{N_F}F(s)$ is an entire function and is of order $1$. Hence by the Hadamard factorization theorem, the lemma is valid.
\end{proof}
\begin{lem}\label{LEM3.1}
Let $F(s)$ satisfy \eqref{condition}. 
For sufficiently large $T$, we have
\begin{align*}
N_{F(s)}(T,T+1)\ll\log T.
\end{align*}
\end{lem}
\begin{proof}
By Lemma \ref{lem2.2}, we have
\begin{align*}
F(s)-\frac{\eta_{n_F}}{n_F^s}=O\left((n_F+1)^{-\sigma}\right).
\end{align*}
Thus there exists a constant $D\ge E_{2F}$ such that $|F(s)-\eta_{n_F}/n_F^s|\leq |\eta_{n_F}/n_F^s|/2$ holds for $\sigma\ge D$. By the triangle inequality, we have 
\begin{align}
|F(s)|\geq |\eta_{n_F}/n_F^s|/2\label{EstOnD}
\end{align}
for $\sigma\ge D$. By \cite[(12)]{on}, we have $F(s)\ll|t|^{\deg_1(F(s))\mu(\sigma)+\varepsilon}$ if $\mu(\sigma)$ satisfies $\zeta(s)\ll |t|^{\mu(\sigma)+\varepsilon}$. This function $\mu(\sigma)$ satisfies the inequality
\begin{align*}
\mu(\sigma)\leq\begin{cases}
  \displaystyle0&(\sigma\geq1),\\
  \displaystyle1/2-\sigma/2&(0<\sigma<1),\\
  \displaystyle1/2-\sigma&(\sigma\leq0).
\end{cases}
\end{align*}
By Jensen's theorem, we have
\begin{align*}
&\int_0^{D-E_{1F}+2}\frac{n(r)}{r}dr\\
&\quad=\frac{1}{2\pi}\int_0^{2\pi}\log\left|F\left(D+iT+\left(D-E_{1F}+2\right)e^{i\theta}\right)\right|d\theta
-\log\left|F\left(D+iT\right)\right|,
\end{align*}
where $n(r)$ is the number of zeros of $F(s)$  counted with multiplicity  in the circle with center $D+iT$ and radius $r$. Since $F(s)\ll |t|^{\deg_1(F(s))\mu(\sigma)+\varepsilon}$, there exists a constant $D_1>0$ such that
\begin{align*}
\log\left|F\left(D+iT+(D-E_{1F}+2)e^{i\theta}\right)\right|\leq D_1\log T.
\end{align*}
Since $|F(D+iT)|\geq |\eta_{n_F}/n_F^D|/2$, there exists a constant $D_2$ such that
\begin{align*}
\log\left|F\left(D+iT\right)\right|\ge D_2.
\end{align*}
Because of $\int_0^{D-E_{1F}+2}(n(r)/r)dr\ge0$, we have
\begin{align}
\int_0^{D-E_{1F}+2}\frac{n(r)}{r}dr\ll\log T.\label{nt1}
\end{align}
On the other hand, we have
\begin{align}
\notag\int_0^{D-E_{1F}+2}\frac{n(r)}{r}dr
\notag&\geq\int_{D-E_{1F}+1}^{D-E_{1F}+2}\frac{n(r)}{r}dr\\
\label{nt2}&\geq n(D-E_{1F}+1)\int_{D-E_{1F}+1}^{D-E_{1F}+2}\frac{1}{r}dr.
\end{align}
From (\ref{nt1}) and (\ref{nt2}), we have
\begin{align*}
N_{F(s)}(T,T+1)\leq n(D-E_{1F}+1)\ll\log T.
\end{align*}
\end{proof}
\begin{lem}\label{LEM3.2}
Let $F(s)$ satisfy \eqref{condition}. Let $\sigma_1$ and $\sigma_2$ be real numbers with $\sigma_1<\sigma_2$. For $s\in\mathbb{C}$ with $\sigma_1<\sigma<\sigma_2$ and large $t$, we have
\begin{align*}
\frac{F'(s)}{F(s)}=\sum_{|\gamma_F-t|<1}\frac{1}{s-\rho_F}+O(\log t).
\end{align*}
\end{lem}
\begin{proof}
Similar to estimate \eqref{EstOnD}, we have
\begin{align*}
|F(E+it)|\geq \frac{1}{2}\frac{|\eta_{n_F}|}{n_F^E}
\end{align*}
and
\begin{align*}
|F'(E+it)|\leq 2\frac{|\eta_{n_{F'}}|}{n_{F'}^E}
\end{align*}
for sufficiently large $E$.
Hence, we have
\begin{align*}
\frac{F'(E+it)}{F(E+it)}\leq4\frac{|\eta_{n_{F'}}|/n_{F'}^E}{|\eta_{n_F}|/n_F^E}\ll1.
\end{align*}
By Lemma \ref{LEM4}, we have
\begin{align*}
&\frac{F'(s)}{F(s)}=\sum_{\substack{\rho_F\neq0\\ \rho_F\mbox{\footnotesize{:\ zeros of }}F(s)}}\left(\frac{1}{s-\rho_F}+\frac{1}{\rho_F}\right)+O(\log t)
\end{align*}
for  $\sigma_1<\sigma<\sigma_2$ and large $t$. Substituting $s=E+it$, we have
\begin{align*}
O(\log t)=\sum_{\substack{\rho_F\neq0\\ \rho_F\mbox{\footnotesize{:\ zeros of }}F(s)}}\left(\frac{1}{E+it-\rho_F}+\frac{1}{\rho_F}\right).
\end{align*}
Therefore we have
\begin{align*}
\frac{F'(s)}{F(s)}&=\sum_{\substack{\rho_F\neq0\\ \rho_F\mbox{\footnotesize{:\ zeros of }}F(s)}}\left(\frac{1}{s-\rho_F}+\frac{1}{E+it-\rho_F}\right)+O(\log t)\\
&=\left(\sum_{|\gamma_F-t|<1}+\sum_{\substack{\rho_F\neq0\\|\gamma_F-t|\ge1}}\right)\left(\frac{1}{s-\rho_F}+\frac{1}{E+it-\rho_F}\right)+O(\log t).
\end{align*}
For $\sigma_1<\sigma<\sigma_2$ and $|\gamma_F-t|\ge1$, we have
\begin{align*}
\left|\frac{1}{s-\rho_F}+\frac{1}{E+it-\rho_F}\right|\ll\frac{1}{|t-\gamma_F|^2}
\end{align*}
Hence by Lemma \ref{LEM3.1}, we have
\begin{align*}
\sum_{\substack{\rho_F\neq0\\|\gamma_F-t|\ge1}}\left(\frac{1}{s-\rho_F}+\frac{1}{E+it-\rho_F}\right)=O(\log t).
\end{align*}
By Lemma \ref{LEM3.1}, we also have
\begin{align*}
\sum_{|\gamma_F-t|<1}\frac{1}{E+it-\rho_F}=O(\log t).
\end{align*}
\end{proof}
\begin{lem}\label{lem7}
Let $F(s)$ satisfy \eqref{condition}. For sufficiently large $\sigma$ and $|t|\ge1$, we have
\begin{align*}
\frac{F'}{F}(1-s)=O(|\log s|).
\end{align*}
\end{lem}
\begin{proof}
Since
\begin{align*}
&\frac{d}{ds}\left(\zeta^{(0)}(s)^{d_{0j}}\zeta^{(1)}(s)^{d_{1j}}\cdots\zeta^{(k)}(s)^{d_{kj}}\right)\\
&\qquad\qquad=\sum_{a=0}^k d_{aj}\zeta^{(a)}(s)^{d_{aj}-1}\zeta^{(a+1)}(s)\prod_{\substack{0\le l\le k\\l\neq a}}\zeta^{(l)}(s)^{d_{lj}},
\end{align*}
we have $\deg_{1}(F'(s))\le\deg_{1}(F(s))$ and $\deg_{2}(F'(s))\le\deg_{2}(F(s))+1$. By Lemma \ref{lem1}, we have
\begin{align}
F'(1-s)=O\left(|\log s|^{\deg_{2}\left(F(s)\right)+1}\left|(2\pi)^{-s}\Gamma(s)\cos\frac{\pi s}{2}\ \zeta(s)\right|^{\deg_{1}\left(F(s)\right)}\right).\label{vmt}
\end{align}
Note that we do not assume the condition \eqref{condition} for $F'(s)$. If $F'(s)$ does not satisfy this condition, the main term in Lemma \ref{lem1} vanishes. Then, $F'(1-s)$ can be bounded by the error term in Lemma \ref{lem1}. Therefore \eqref{vmt} also holds in this case. Dividing \eqref{vmt} by $F(1-s)$ and applying Lemma \ref{lem1}, we obtain Lemma \ref{lem7}.
\end{proof}


\section{Proof of Theorem \ref{1}}

In this section, we prove Theorem \ref{1}. By Cauchy's theorem, we have
\begin{align*}
N_{F(s)}(1,T)&=\frac{1}{2\pi}\Im\left(\int_{E_{1F}'+i}^{E_{2F}'+i}+\int_{E_{2F}'+i}^{E_{2F}'+iT}+\int_{E_{2F}'+iT}^{E_{1F}'+iT}+\int_{E_{1F}'+iT}^{E_{1F}'+i}\right)\frac{F'(s)}{F(s)}ds\\
&=:\frac{1}{2\pi}(J_1+J_2+J_3+J_4),
\end{align*}
for sufficiently large $E_{2F}'\geq E_{2F}$ and sufficiently small $E_{1F}'\leq E_{1F}$.

The first term $J_1$ does not depend on $T$, so we have $J_1=O(1)$. 

By Lemma \ref{lem2.2}, the second term is estimated as
\begin{align*}
\frac{1}{2\pi}J_2=\frac{1}{2\pi} \left[\arg(F(s))\right]_{E_{2F}'+i}^{E_{2F}'+iT}=\frac{1}{2\pi} \left[\arg(\eta_{n_F}n_F^{-s})+O(1)\right]_{E_{2F}'+i}^{E_{2F}'+iT}.
\end{align*}
Hence we have
\begin{align*}
\frac{1}{2\pi}J_2=-\frac{T}{2\pi}\log n_F+O(1).
\end{align*}

Next, we estimate $J_3$. Applying Lemma \ref{LEM3.2}, we have
\begin{align*}
J_3&=\Im\int_{E_{2F}'+iT}^{E_{1F}'+iT}\sum_{|\gamma_F-t|<1}\frac{1}{s-\rho_F}ds+O\left(\int_{E_{2F}'+iT}^{E_{1F}'+iT}\log t\ ds\right)\\
&=\Im\sum_{|\gamma_F-T|<1}\int_{E_{2F}'+iT}^{E_{1F}'+iT}\frac{1}{s-\rho_F}ds+O\left(\log T\right).
\end{align*}
For each integral, we change the path of integration. If $\gamma_F< T$, then we change the path to the upper semicircle with center $\rho_F$ and radius $1$. If $\gamma_F> T$, then we change the path to the lower semicircle with center $\rho_F$ and radius $1$. Then all integrals are bounded by $O(1)$.
Therefore, by Lemma \ref{LEM3.1}, we have
\begin{align*}
J_3=\sum_{|\gamma_F-T|<1}O(1)+O\left(\log T\right)=O(\log T).
\end{align*}

Finally, we consider $J_4$. Since $\arg(F(E_{1F}'+i))$ does not depend on $T$, we have
\begin{align*}
J_4= \left[\arg(F(s))\right]_{E_{1F}'+iT}^{E_{1F}'+i}=-\arg\left(F(E_{1F}'+iT)\right)+O(1).
\end{align*}
By Lemma \ref{lem1}, we have
\begin{align*}
&\arg\left(F(E_{1F}'+iT)\right)\\
&\quad=\deg_1(F(s))T\log(2\pi)+\deg_1(F(s))\arg\Gamma(1-E_{1F}'-iT)+O(1)
\end{align*}
By Stirling's formula, we have
$$
\arg\Gamma(1-E_{1F}'-iT)=-T\log\frac{T}{e}+O(\log T).
$$
Hence we have
\begin{align*}
J_4=\deg_1(F(s))T\log\frac{T}{2\pi e}+O(\log T).
\end{align*}

Therefore we have
\begin{align*}
N_{F(s)}(1,T)&=\frac{1}{2\pi}(J_1+J_2+J_3+J_4)\\
&=\frac{\deg_1(F(s))T}{2\pi}\log\frac{T}{2\pi e}-\frac{T}{2\pi}\log n_F+O(\log T).
\end{align*} {\hfill $\square$}


\section{Proofs of Theorem \ref{thm4} and Theorem \ref{thm5}}

\begin{lem}\label{LEM5.1}
For $U\gg1$, sufficiently large $T\gg U$ and a real number $b<D$
 ($D$ is defined in the proof of Lemma \ref{LEM3.1}), we have
\begin{align*}
2\pi \sum_{\substack{T<\gamma_F<T+U\\\beta_F>b}}\left(\beta_F-b\right)=\int_T^{T+U}\log\left|F\left(b+it\right)\right|dt-U\log\left|\frac{\eta_{n_F}}{n_F^{b}}\right|+O(\log T).
\end{align*}
\end{lem}
\begin{proof}
For a real number $b$ with $b<D$, by Littlewood's lemma, we have
\begin{align}
\notag2\pi &\sum_{\substack{T<\gamma_F<T+U\\\beta_F>b}}\left(\beta_F-b\right)\\
\notag&=\int_T^{T+U}\log\left|\frac{F\left(b+it\right)}{\eta_{n_F}/n_F^{b+it}}\right|dt-\int_T^{T+U}\log\left|\frac{F\left(D+it\right)}{\eta_{n_F}/n_F^{D+it}}\right|dt\\
\label{Lit}&+\int_{b}^D\arg\frac{F\left(\sigma+i(T+U)\right)}{\eta_{n_F}/n_F^{\sigma+i(T+U)}}d\sigma-\int_{b}^D\arg\frac{F\left(\sigma+iT\right)}{\eta_{n_F}/n_F^{\sigma+iT}}d\sigma
\end{align}
where we take the logarithmic branch of $\arg (F\left(s\right)/(\eta_{n_F}/n_F^{s}))$ as $\arg (F\left(s\right)/(\eta_{n_F}/n_F^{s}))\to0$ as $\sigma\to\infty$. We define the function $G(s)$ by
\begin{align*}
G(s):=\frac{F\left(s\right)}{\eta_{n_F}/n_F^{s}}.
\end{align*}
Because of the choice of $D$, we have $\Re(G(s))\ge1/2$ for $\sigma\geq D$. Furthermore we define $H_{T}(s)$ by
\begin{align*}
H_{T}(s):=\frac{G(s+iT)+\overline{G(\overline{s}+iT)}}{2}.
\end{align*}
Then we have $H_{T}(\sigma)=\Re(G(\sigma+iT))$. When $H_T(\sigma)$ has $n$ zeros in $\sigma\in[b,D]$, we can estimate the argument 
\begin{align*}
|\arg G(\sigma+iT)|\leq \pi n+O(1)
\end{align*}
for $\sigma\in[b,D]$. Thus we estimate the number of zeros of $H_T(\sigma)$.
Let $n_D'(r)$ denote the number of zeros of $H_{T}(s)$  in the circle with center $D$ and radius $r$. By Jensen's theorem, we have
\begin{align*}
&\int_0^{D-b+1}\frac{n_{D}'(r)}{r}dr\\
&=\frac{1}{2\pi}\int_{0}^{2\pi}\log\left|H_{T}\left(D+\left(D-b+1\right)e^{i\theta}\right)\right|d\theta-\log\left|H_{T}\left(D\right)\right|
\end{align*}
As mentioned in the proof of Lemma \ref{LEM3.1}, we have $F(s)\ll|t|^{\deg_1(F(s))\mu(\sigma)+\varepsilon}$, so we also have $H_T(s)\ll|t+T|^{\deg_1(F(s))\mu(\sigma)+\varepsilon}$. Hence we have
\begin{align*}
\int_0^{D-b+1}\frac{n_{D}'(r)}{r}dr= O(\log T).
\end{align*}
On the other hand, we have
\begin{align*}
\int_0^{D-b+1}\frac{n_{D}'(r)}{r}dr\ge \int_{D-b}^{D-b+1}\frac{n_{D}'(r)}{r}dr\ge n_{D}'(D-b)\int_{D-b}^{D-b+1}\frac{1}{r}dr.
\end{align*}
Hence finally we obtain
\begin{align*}
|\arg G(\sigma+iT)|\ll  n_{D}'(D-\sigma)\le n_{D}'(D-b)\ll \log T
\end{align*}
for $\sigma\in[b,D]$.
From this estimate, we can bound the third and fourth terms of  (\ref{Lit}) by $O(\log T)$.

Now, we estimate the second term of  (\ref{Lit}).  By Cauchy's integral formula, we have
\begin{align}
\notag&\left|\int_T^{T+U}\log G\left(D+it\right) dt\right|\\
\label{RightEst}&\leq\left|\int_D^{V}\log G\left(\sigma+iT\right) d\sigma\right|+
\left|\int_T^{T+U}\log G\left(V+it\right) dt\right|+
\left|\int_V^{D}\log G\left(\sigma+i(T+U)\right) d\sigma\right|
\end{align}
for sufficiently large $V>D$. By Lemma \ref{lem2.2}, we have 
$$\log G(\sigma+it)=O\left((1+1/n_F)^{-\sigma}\right).$$
Thus we have
\begin{align*}
\int_T^{T+U}\log G\left(V+it\right) dt\ll U(1+1/n_F)^{-V} \to0
\end{align*}
as $V\to\infty$. The first term of \eqref{RightEst} is estimated as 
\begin{align*}
\int_D^{\infty}\log G\left(\sigma+iT\right) d\sigma\ll (1+1/n_F)^{-D} \ll1.
\end{align*}
Similarly, The third term of \eqref{RightEst} is also bounded by $O(1)$. Hence  the second term of  (\ref{Lit}) is $O(1)$.
\end{proof}

\begin{lem}\label{LEM5.2}
Let $\alpha>1/2$. For $T^\alpha\leq U\leq T$, we have
\begin{align*}
2\pi \sum_{\substack{T<\gamma_F<T+U\\\beta_F>1/2}}\left(\beta_F-\frac{1}{2}\right)\ll U\log\log T.
\end{align*}
\end{lem}
\begin{proof}
By the previous lemma, it is enough to prove that there exists a constant $A$ such that
\begin{align*}
\int_T^{T+U}\log\left|F\left(\frac{1}{2}+it\right)\right|dt\leq AU\log\log T.
\end{align*}
We estimate $|F(s)|$ as a product of sum of zeta functions;
\begin{align*}
|F(s)|&\leq \sum_{j=1}^{M} |c_j|\left|\zeta^{(0)}(s)\right|^{d_{0j}}\left|\zeta^{(1)}(s)\right|^{d_{1j}}\cdots\left|\zeta^{(k)}(s)\right|^{d_{kj}}\\
&\leq \left(\sum_{j=1}^{M} |c_j|\right)\left(1+\sum_{l=0}^{k} \left|\zeta^{(l)}(s)\right|\right)^{\deg_1(F(s))}.
\end{align*}
Hence we have
\begin{align*}
\int_T^{T+U}\log\left|F\left(\frac{1}{2}+it\right)\right|dt\leq\deg_1(F(s))\int_T^{T+U}\log\left(1+\sum_{l=0}^{k} \left|\zeta^{(l)}\left(\frac{1}{2}+it\right)\right|\right)dt+O(U).
\end{align*}
For any positive numbers $v_0,\ldots,v_k$, we can bound
\begin{align*}
\leq\frac{\deg_1(F(s))}{v_{\min}}\int_T^{T+U}\log\max\left(1,\left|\zeta^{(0)}\left(\frac{1}{2}+it\right)\right|^{v_0},\ldots,\left|\zeta^{(k)}\left(\frac{1}{2}+it\right)\right|^{v_k}\right)dt+O(U)
\end{align*}
where $v_{\min}:=\min\{v_0,\ldots,v_k\}$. By  Jensen's inequality, we have
\begin{align*}
&\leq\frac{\deg_1(F(s))}{v_{\min}}U\log\left(\frac{1}{U}\int_T^{T+U}\max\left(1,\left|\zeta^{(0)}\left(\frac{1}{2}+it\right)\right|^{v_0},\ldots,\left|\zeta^{(k)}\left(\frac{1}{2}+it\right)\right|^{v_k}\right)dt\right)+O(U)\\
&\leq\frac{\deg_1(F(s))}{v_{\min}}U\\
&\qquad\log\left(1+\frac{1}{U}\int_T^{T+U}\left|\zeta^{(0)}\left(\frac{1}{2}+it\right)\right|^{v_0}dt+\cdots+\frac{1}{U}\int_T^{T+U}\left|\zeta^{(k)}\left(\frac{1}{2}+it\right)\right|^{v_k}dt\right)+O(U).
\end{align*}
Hence, it remains to prove that there exists a positive number $v_j$ such that
\begin{align}
\int_T^{T+U}\left|\zeta^{(j)}\left(\frac{1}{2}+it\right)\right|^{v_j}dt\ll U\log T\label{AUlogT}
\end{align}
for each $0\le j\le k$.
By \cite[Claim]{ki} and \cite[Theorem 7.4]{ti}, we have
\begin{align*}
\int_T^{T+U}\left|\frac{\zeta^{(j)}}{\zeta}\left(\frac{1}{2}+it\right)\right|^{1/(2j)}dt\ll U\sqrt{\log T}
\end{align*}
and
\begin{align*}
\int_T^{T+U}\left|\zeta\left(\frac{1}{2}+it\right)\right|^2dt\ll U\log T.
\end{align*}
Applying H\"older's inequality
\begin{align*}
&\int_T^{T+U}\left|\zeta^{(j)}\left(\frac{1}{2}+it\right)\right|^{v_j}dt\\
&\qquad\leq \left(\int_T^{T+U}\left|\frac{\zeta^{(j)}}{\zeta}\left(\frac{1}{2}+it\right)\right|^{pv_j}dt\right)^{1/p}\left(\int_T^{T+U}\left|\zeta\left(\frac{1}{2}+it\right)\right|^{qv_j}dt\right)^{1/q}
\end{align*}
with $v_j=2/(4j+1),\ p=1+1/(4j)$ and $q=4j+1$,
we have
\begin{align*}
&\int_T^{T+U}\left|\zeta^{(j)}\left(\frac{1}{2}+it\right)\right|^{v_j}dt\ll \left(U\sqrt{\log T}\right)^{1/p}\left(U\log T\right)^{1/q}\ll U\log T.
\end{align*}
Thus we obtain \eqref{AUlogT}.
\end{proof}

\begin{lem}\label{thm5-1}
Let $\alpha>1/2$ and $T^\alpha\leq U\leq T$. Then we have
\begin{align*}
N_{F(s)}^+\left(\frac{1}{2}+\delta;T,T+U\right)=O\left(\frac{U\log\log T}{\delta}\right)
\end{align*}
for $\delta>0$ uniformly.
\end{lem}
\begin{proof}
From Lemma \ref{LEM5.2}, we have
\begin{align*}
\sum_{\substack{T<\gamma_F<T+U\\\beta_F>1/2+\delta}}\left(\beta_F-\frac{1}{2}\right)\leq\sum_{\substack{T<\gamma_F<T+U\\\beta_F>1/2}}\left(\beta_F-\frac{1}{2}\right)\ll U\log\log T.
\end{align*}
On the other hand, we have
\begin{align*}
\sum_{\substack{T<\gamma_F<T+U\\\beta_F>1/2+\delta}}\left(\beta_F-\frac{1}{2}\right)\geq \delta N_{F(s)}^+\left(\frac{1}{2}+\delta;T,T+U\right).
\end{align*}
By these estimates, we obtain the lemma.
\end{proof}

\begin{lem}\label{LEM5.4}
Let $F(s)$ satisfy \eqref{condition} and $\alpha>1/2$. For sufficiently small $b$, large $T$ and $T^\alpha\leq U\leq T$, we have
\begin{align*}
&2\pi\sum_{T<\gamma_F<T+U}\left(\beta_F-b\right)\\
&\qquad=\deg_1(F(s))\left(\frac{1}{2}-b\right)\left\{(T+U)\log\frac{T+U}{2\pi}-T\log\frac{T}{2\pi}-U\right\}\\
&\qquad+\deg_2(F(s))U\log\log T+U\log \left|\frac{\sum_{j\in J}c_j}{\eta_{n_F}n_F^{-b}}\right|+O\left(\frac{U}{\log T}\right).
\end{align*}
\end{lem}
\begin{proof}
We use Lemma \ref{LEM5.1}. By Lemma \ref{lem1}, the integrand in Lemma \ref{LEM5.1} can be calculated as
\begin{align}
\notag&\log\left|F(b+it)\right|=\log \left|\sum_{j\in J}c_j\right|+\deg_2(F(s))\log|\log(1-b-it)|\\
&+\deg_1(F(s))\log|\chi(b+it)|+\deg_1(F(s))\log|\zeta(1-b-it)|+O\left(\frac{1}{\log(1-b-it)}\right)\label{b1}
\end{align}
where $\chi(s)=2^s\pi^{-1+s}\sin(\pi s/2)\Gamma(1-s)$.
By equations
\begin{align*}
\log|\log(1-b-it)|=\log\log t+O\left(\frac{1}{\log t}\right)
\end{align*}
and
\begin{align*}
\log|\chi(s)|=\left(\frac{1}{2}-\sigma\right)\log\left|\frac{t}{2\pi}\right|+O\left(\frac{1}{t}\right),
\end{align*}
we have
\begin{align*}
\int_T^{T+U}\log\left|F(b+it)\right|dt&=U\log \left|\sum_{j\in J}c_j\right|+\deg_2(F(s))\left\{(T+U)\log\log(T+U)-T\log\log T\right\}\\
&+\deg_1(F(s))\left(\frac{1}{2}-b\right)\left\{(T+U)\log\frac{T+U}{2\pi}-T\log\frac{T}{2\pi}-U\right\}\\
&+\deg_1(F(s))\int_T^{T+U}\log|\zeta(1-b-it)|dt+O\left(\frac{U}{\log T}\right).
\end{align*}
Similar to the estimate of the second term of \eqref{Lit}, we can  estimate
\begin{align*}
\int_T^{T+U}\log|\zeta(1-b-it)|dt\ll1.
\end{align*}
Moreover we can easily check
\begin{align*}
T(\log\log(T+U)-\log\log T)=O\left(\frac{U}{\log T}\right).
\end{align*}
Applying Lemma \ref{LEM5.1}, we obtain Lemma \eqref{LEM5.4}.
\end{proof}

\noindent{\bf (Proof of Theorem \ref{thm4})}\ 
By Theorem \ref{1} and Lemma \ref{LEM5.4}, we have
\begin{align*}
2\pi N_{F(s)}(T,T+U)=\deg_1(F(s))\left\{(T+U)\log\frac{T+U}{2\pi e}-T\log\frac{T}{2\pi e}\right\}-U\log n_F+O(\log T)
\end{align*}
and
\begin{align*}
&2\pi\sum_{T<\gamma_F<T+U}\left(\beta_F-b\right)\\
&\qquad=\deg_1(F(s))\left(\frac{1}{2}-b\right)\left\{(T+U)\log\frac{T+U}{2\pi e}-T\log\frac{T}{2\pi e}\right\}\\
&\qquad+\deg_2(F(s))U\log\log T+U\log \left|\frac{\sum_{j\in J}c_j}{\eta_{n_F}n_F^{-b}}\right|+O\left(\frac{U}{\log T}\right).
\end{align*}
Hence we have
\begin{align*}
&2\pi\sum_{T<\gamma_F<T+U}\left(\beta_F-\frac{1}{2}\right)\\
&=2\pi\sum_{T<\gamma_F<T+U}\left(\beta_F-b\right)-\left(\frac{1}{2}-b\right)2\pi N_{F(s)}(T,T+U)\\
&=\deg_2(F(s))U\log\log T+U\log \left|\frac{\sum_{j\in J}c_j}{\eta_{n_F}n_F^{-1/2}}\right|+O\left(\frac{U}{\log T}\right).
\end{align*} {\hfill $\square$}
\begin{lem}\label{thm5-2}
Let $F(s)$ satisfy \eqref{condition}, $\alpha>1/2$, and $T^\alpha\leq U\leq T$. Then we have
\begin{align*}
N_{F(s)}^-\left(\frac{1}{2}-\delta;T,T+U\right)=O\left(\frac{U\log\log T}{\delta}\right)
\end{align*}
for $\delta>0$ uniformly.
\end{lem}
\begin{proof}
We decompose the summation as
\begin{align*}
&2\pi \sum_{T<\gamma_F<T+U}\left(\beta_F-b\right)\\
&=2\pi \sum_{\substack{T<\gamma_F<T+U\\\beta_F>1/2+\delta}}\left\{\left(\beta_F-\frac{1}{2}\right)+\left(\frac{1}{2}-b\right)\right\}\\
&+2\pi \sum_{\substack{T<\gamma_F<T+U\\1/2-\delta\leq\beta_F\leq1/2+\delta}}\left\{\left(\beta_F-\frac{1}{2}\right)+\left(\frac{1}{2}-b\right)\right\}\\
&+2\pi \sum_{\substack{T<\gamma_F<T+U\\\beta_F<1/2-\delta}}\left(\beta_F-b\right).
\end{align*}
By Lemma \ref{LEM5.2}, we have
\begin{align*}
2\pi &\sum_{T<\gamma_F<T+U}\left(\beta_F-b\right)\leq O(U\log\log T)+2\pi \left(\frac{1}{2}-b\right)N_{F(s)}^+\left(\frac{1}{2}+\delta;T,T+U\right)\\
&+2\pi \left(\frac{1}{2}-b\right)\left(N_{F(s)}(T,T+U)-N_{F(s)}^+\left(\frac{1}{2}+\delta;T,T+U\right)-N_{F(s)}^-\left(\frac{1}{2}-\delta;T,T+U\right)\right)\\
&+2\pi \left(\frac{1}{2}-\delta-b\right)N_{F(s)}^-\left(\frac{1}{2}-\delta;T,T+U\right).
\end{align*}
By Theorem \ref{1},
we have
\begin{align}
\notag2\pi \sum_{T<\gamma_F<T+U}\left(\beta_F-b\right)
&\leq \deg_1(F(s))\left(\frac{1}{2}-b\right)U\log T\\
&-2\pi \delta N_{F(s)}^-\left(\frac{1}{2}-\delta;T,T+U\right)+O(U\log\log T).\label{No1}
\end{align}
By Lemma \ref{LEM5.4}, the left-hand side is
\begin{align}
2\pi\sum_{T<\gamma_F<T+U}\left(\beta_F-b\right)=\deg_1(F(s))\left(\frac{1}{2}-b\right)U\log T+O\left(U\log\log T\right).\label{No2}
\end{align}
From (\ref{No1}) and (\ref{No2}), we have
\begin{align*}
&N_{F(s)}^-\left(\frac{1}{2}-\delta;T,T+U\right)\leq O\left(\frac{U\log\log T}{\delta}\right).
\end{align*}
Since $N_{F(s)}^-\left(1/2-\delta;T,T+U\right)$ is positive, we finally obtain
\begin{align*}
&N_{F(s)}^-\left(\frac{1}{2}-\delta;T,T+U\right)= O\left(\frac{U\log\log T}{\delta}\right).
\end{align*}
\end{proof}
By Lemma \ref{thm5-1} and Lemma \ref{thm5-2}, we obtain Theorem \ref{thm5}.


\section{Proofs of Theorem \ref{6} and Corollary \ref{cor7}}

\noindent{\bf (Proof of Theorem \ref{6})}\ 
For sufficiently large $U,V>0$ and $U>c'>c>1$, by Cauchy's integral formula, we have
\begin{align*}
&\sum_{1<\gamma_F<T}x^{\rho_F}\\
&=\frac{1}{2\pi i}\left(\int_{-U+i}^{-c'+i}+\int_{-c'+i}^{V+i}+\int_{V+i}^{V+iT}+\int_{V+iT}^{-c'+iT}+\int_{-c'+iT}^{-U+iT}+\int_{-U+iT}^{-U+i}\right)x^s\frac{F'}{F}(s)ds\\
&=:\frac{1}{2\pi i}(K_1+K_2+K_3+K_4+K_5+K_6).
\end{align*}
Note that the constant $c$ is defined in the statement of Lemma \ref{lem1}.

The second term $K_2$ does not depend on $T$, so we have $K_2=O(1)$.

By using the series expansion \eqref{alpFF}, we have
\begin{align*}
K_3&=\sum_{d\in n_F^{-n}\mathbb{N}\ (n\in\mathbb{N})}\alpha_{(F'/F)(s)}(d)\int_{V+i}^{V+iT}\left(\frac{x}{d}\right)^sds\\
&=Ti\alpha_{(F'/F)(s)}(x)+\sum_{d\neq x}\alpha_{(F'/F)(s)}(d)\left[\frac{(x/d)^s}{\log(x/d)}\right]_{V+i}^{V+iT}+O(1)\\
&=Ti\alpha_{(F'/F)(s)}(x)+O(1).
\end{align*}

Similar to the estimate on $J_3$ in Section 3, the fourth term can be bounded by $K_4=O(\log T)$.

Applying Lemma \ref{lem7}, we have
\begin{align*}
K_5=O\left(\int_{-c'}^{-U}x^\sigma |\log(1-\sigma-iT)||d\sigma|\right)=O(\log T),
\end{align*} 
\begin{align*}
K_1=O(\log T),
\end{align*}
and
\begin{align*}
K_6=O\left(x^{-U}T|\log (1+U-iT)|\right)
\end{align*}
uniformly for large $U$. Hence taking $U\to\infty$, we obtain $K_1+K_5+K_6=O(\log T)$.

{\hfill $\square$}

\noindent{\bf (Proof of Corollary \ref{cor7})}\ 
To prove Corollary \ref{cor7}, we use Weyl's criterion;
\begin{lem}[Weyl's criterion] \label{WY}
A sequence $\{x_n\}\subset\mathbb{R}$ is uniformly distributed modulo one if, and only if, for any integer $m\neq0$,
\begin{align*}
\lim_{N\to\infty}\frac{1}{N}\sum_{n=1}^{N}e^{2\pi i m x_n}=0.
\end{align*}
\end{lem}
See \cite[Theorem 2.1]{ku} for the proof.

By Theorem \ref{thm1} and Theorem \ref{thm5} with $\delta=(\log\log T)^2/\log T$ and $U=T$, we have
\begin{align}
\notag\sum_{T<\gamma_F<2T}\left|\beta_F-\frac{1}{2}\right|&=\left(\sum_{\substack{T<\gamma_F<2T\\|1/2-\beta_F|\le \delta}}+\sum_{\substack{T<\gamma_F<2T\\|1/2-\beta_F|> \delta}}\right)\left|\beta_F-\frac{1}{2}\right|\\
&=O\left(T(\log\log T)^2\right)+O\left(T\frac{\log T}{\log\log T}\right).\label{add2}
\end{align} 
Substituting $2^{-l}T$ for $T$ and adding \eqref{add2} all $l=1,2,\ldots,l_0$ with suitable $l_0$, we obtain
\begin{align*}
\sum_{1<\gamma_F<T}\left|\beta_F-\frac{1}{2}\right|&=O\left(T\frac{\log T}{\log\log T}\right)=o\left(N_{F(s)}(1,T)\right).
\end{align*} 
Since
\begin{align*}
e^y-1=\int_0^ye^tdt=O\left(|y|\max\{1,e^y\}\right)
\end{align*}
for $y\in\mathbb{R}$, we have
\begin{align*}
|x^{1/2+i\gamma_F}-x^{\beta_F+i\gamma_F}|\le x^{\beta_F}\left|e^{(1/2-\beta_F)\log x}-1\right|\ll\left|\beta_F-\frac{1}{2}\right|\max\{x^{\beta_F},x^{1/2}\}|\log x|.
\end{align*}
Hence for each $x>1$, we have
\begin{align*}
\frac{1}{N_{F(s)}(1,T)}\sum_{1<\gamma_F<T}|x^{1/2+i\gamma_F}-x^{\beta_F+i\gamma_F}|
&\ll \frac{x^{E_{2F}}|\log x|}{N_{F(s)}(1,T)}\sum_{1<\gamma_F<T}\left|\beta_F-\frac{1}{2}\right|\\
&=O\left(\frac{1}{\log\log T}\right).
\end{align*}
Thus by Theorem \ref{6}, for each $x>1$, we have
\begin{align*}
\frac{1}{N_{F(s)}(1,T)}\sum_{1<\gamma_F<T}x^{i\gamma_F}&=\frac{1}{N_{F(s)}(1,T)}x^{-1/2}\sum_{1<\gamma_F<T}x^{1/2+i\gamma_F}\\
&=\frac{1}{N_{F(s)}(1,T)}x^{-1/2}\sum_{1<\gamma_F<T}x^{\beta_F+i\gamma_F}+O\left(\frac{1}{\log\log T}\right)\\
&=O\left(\frac{1}{\log\log T}\right).
\end{align*}
Substituting $x=e^{2\pi m\alpha}$, we obtain
\begin{align*}
\frac{1}{N_{F(s)}(1,T)}\sum_{1<\gamma_F<T}e^{2\pi im\alpha\gamma_F}=O\left(\frac{1}{\log\log T}\right),
\end{align*}
when $m\alpha$ is positive. Considering complex conjugate, we can also estimate the case  when $m\alpha$ is negative.
Hence we obtain Corollary \ref{cor7}.

{\hfill $\square$}


Multiple Zeta Research Center, Kyushu University\\
Motooka, Nishi-ku, Fukuoka 819-0395\\
 Japan\\
E-mail: t-onozuka@math.kyushu-u.ac.jp

\end{document}